\documentclass[10pt, english]{amsart}

\usepackage{amsmath,amssymb,enumerate}

\usepackage[T1]{fontenc}
\usepackage[all,cmtip]{xy}

\usepackage{babel}
\usepackage{amstext}
\usepackage{amsmath}
\usepackage{amsfonts}
\usepackage{latexsym}
\usepackage{ifthen}

\usepackage{xypic}
\xyoption{all}
\newcommand{\Bs}{{\rm Bs}}

\newtheorem{lemma1}{}[section]

\newenvironment{lemma}{\begin{lemma1}{\bf Lemma.}}{\end{lemma1}}

\newenvironment{theorem}{\begin{lemma1}{\bf Theorem.}}{\end{lemma1}}

\newenvironment{proposition}{\begin{lemma1}{\bf Proposition.}}{\end{lemma1}}
\newenvironment{corollary}{\begin{lemma1}{\bf Corollary.}}{\end{lemma1}}
\newenvironment{remark}{\begin{lemma1}{\bf Remark.}\rm}{\end{lemma1}}

\newenvironment{definition}{\begin{lemma1}{\bf Definition.}}{\end{lemma1}}

\newenvironment{conjecture}{\begin {lemma1}{\bf Conjecture.}}{\end{lemma1}}

\newenvironment{problem}{\begin{lemma1}{\bf Problem.}}{\end{lemma1}}

\newenvironment{the local obstruction - setup}{\begin{lemma1}{\bf The local obstruction - setup.}}{\end{lemma1}}

\newenvironment{remark*}{{\bf Remark.}}{}
\newenvironment{remarks*}{{\bf Remarks.}}{}
\newenvironment{example*}{{\bf Example.}}{}
\newenvironment{assumption*}{{\bf Assumption.}}{}

\newcommand{\Q}{\ensuremath{\mathbb{Q}}}

\newcommand{\C}{\ensuremath{\mathbb{C}}}
\newcommand{\N}{\ensuremath{\mathbb{N}}}
\newcommand{\PP}{\ensuremath{\mathbb{P}}}

\usepackage{eurosym}

\newcommand{\merom}[3]{\ensuremath{#1:#2 \dashrightarrow #3}}

\newcommand{\holom}[3]{\ensuremath{#1:#2  \rightarrow #3}}
\newcommand{\fibre}[2]{\ensuremath{#1^{-1} (#2)}}

\makeatletter
\ifnum\@ptsize=0  \fi \ifnum\@ptsize=2  \fi 

\newcommand\sF{{\mathcal F}}

\newcommand\sO{{\mathcal O}}

\DeclareMathOperator*{\pic}{Pic}

\usepackage{xcolor}
\usepackage{hyperref}

\author{Andreas H\"oring}
\author{Thomas Peternell}

\address{Andreas H\"oring, Universit\'e C\^ote d'Azur, CNRS, LJAD, France, Institut universitaire de France}
\email{Andreas.Hoering@univ-cotedazur.fr}

\address{Thomas Peternell, Mathematisches Institut, Universit\"at Bayreuth, 95440 Bayreuth, 
Germany}
\email{thomas.peternell@uni-bayreuth.de}


\title{Fano threefolds with affine canonical extensions} 
\date{November 20, 2022} 

\begin{document}

\begin{abstract} 
Let $M$ be a smooth Fano threefold such that a canonical extension of the tangent bundle
is an affine manifold. We show that $M$ is rational homogeneous.
\end{abstract} 

\maketitle


\section{Introduction}

\subsection{Main result}
Consider a compact K\"ahler manifold $M$ with tangent bundle~$T_M$. Any K\"ahler class determines  a so-called 
canonical extension 
$$ 0 \to \sO_M \to V \to T_M \to 0.$$
Then we may consider the manifold $$Z_M = \PP(V)  \setminus \PP(T_M)$$ which is an affine bundle over $M$. In this context, we have the 
following conjecture \cite{GW20}, see also \cite[Conj.1.1]{HP21}: 

\begin{conjecture} \label{MainConj}
Let $M$ be a projective manifold such that a canonical extension $Z_M$ is affine for some K\"ahler class. 
Then $M$ is rational homogeneous, i.e., $M = G/P$ with $G$ a semi-simple complex Lie group and $P$ a parabolic subgroup.
\end{conjecture}

The aim of the paper is to prove Conjecture \ref{MainConj} for Fano threefolds: 

\begin{theorem} \label{thm:Fano3} 
Let $M$ be a smooth Fano threefold. If $Z_M$ is affine for some K\"ahler class, then $M$ is rational homogeneous.
\end{theorem} 

Fano threefolds are of course rather special manifolds, but they are a natural testing ground for Conjecture \ref{MainConj}: since $Z_M$ is affine, the tangent bundle $T_M$ is big \cite[Prop.4.2]{GW20}, a very restrictive property for manifolds that are not rational homogeneous \cite{HLS22}.
Thus it seems likely that potential counterexamples to Conjecture \ref{MainConj} share at least some properties of rational homogeneous spaces, e.g. we would expect them to be rationally connected or even Fano.

To prove Theorem \ref{thm:Fano3}, we first investigate the structure of extremal rays on $M$. We have shown in \cite[Cor.1.8]{HP21} that $M$ does not admit any birational Mori contraction. This will be used, via classification, to settle the case that $b_2(M) \geq 2$. 
The far more difficult case is that $b_2(M) = 1$. The tangent bundle being big, we know  by \cite[Cor.1.2]{HL21} that $M$ must be the 
del Pezzo threefold $V_5$ of degree five, unless $M$ is rational homogeneous (i.e., $\PP_3$ or the quadric $Q_3$). Thus we will prove - which will be surprisingly difficult - 

\begin{theorem} \label{theorem-V5}
Let $M$ be the del Pezzo threefold of degree five. Then $Z_M$ is not Stein.
\end{theorem}

 A more explicit explanation is given in Theorem \ref{theorem-V5-details}.

\subsection{A general strategy} 
When studying complements of hypersurfaces, an interesting source of examples
is given by small modifications \cite[p.2]{HP21}, \cite[Ex.5.5]{Ott12}:
given a smooth ample divisor $Y^- \subset X^-$ in a projective manifold,
one tries to find a subset $C^- \subset Y^- \subset X^-$ of codimension at least two that can be ``flipped'', i.e. there exists a birational map
$$
X^- \dashrightarrow X^+
$$
that is an isomorphism $X^- \setminus C^- \simeq X^+ \setminus C^+$
for some subset $C^+ \subset X^+$ of codimension at least two. 
Let $Y^+  \subset X^+$ be the strict transform of $Y^-$, then $Y^- \subset X^-$
is never an ample divisor. Nevertheless, if $C^+ \subset X^+$, its complement
$$
X^+ \setminus Y^+ \simeq X^- \setminus Y^-
$$
is affine. In order to decide whether a canonical extension $Z_M$ affine or not, our goal is to inverse this procedure: we aim to construct a small birational map 
$$ 
\psi: X=\mathbb P(V) \dasharrow \mathbb P(V)^-=X^-.
$$
which is biholomorphic on $Z_M$ such that the strict transform $ \mathbb P^-(T_M)$ is a big and nef divisor on $\mathbb P^-(V)$. 
There is a natural candidate for the space $X^-$:
the tangent bundle $T_M$ being big, the tautological class $\zeta_V \rightarrow \PP(V)$ is big. Therefore for some $k \gg 0$ we could consider the birational map
$$
\psi : X=\mathbb P(V) \dasharrow X^-
$$
given by the image of $\varphi_{|k \zeta_V|}$. While the indeterminacy locus of $\psi$ is contained in $\PP(T_M) \subset \PP(V)$ it is absolutely not clear whether such a map
is an isomorphism in codimension one. 
Therefore we introduce the following terminology. 

\begin{definition}  \label{definition-weak-fano-model}
Let $M$ be a projective manifold, and let $\PP(T_M) \subset \PP(V)$ be a canonical extension associated to some K\"ahler class on $M$.
Let 
$$\psi: \PP(V) \dasharrow \PP(V)^- $$
be a finite sequence of antiflips (i.e., inverses of flips). Let $\PP(T_M)^-$ be the strict transform of 
$\PP(T_M)$ in $\PP(V)^- $, and assume that $\psi_{\vert \PP(T_M)} $ is again a finite sequence of antiflips. 
In particular, $\PP(V)^-$ and $\PP(T_M)^-$ are $\mathbb Q$-factorial with terminal singularities. 
We say that 
$$
(\PP(V)^-,\PP(T_M)^-))
$$ 
is a weak Fano model of $(\PP(V),\PP(T_M))$, if the anticanonical divisor of $\PP(V)^-$ is nef and big. 
\end{definition} 

Verifying that $\PP(V)^-$ is a weak Fano variety is equivalent to verifying this property
for $\PP(T_M)^-$. In fact, $\PP(T_M) \in \vert \zeta_V \vert  = \vert \sO_{\PP(V)}(1) \vert$. Thus, if $\zeta_V^-$ denotes
the induced Weil divisor class on $\PP(V)^-$, then 
$$ 
\PP(T_M)^- \in \vert \zeta_V^- \vert.
$$ 
Since $-K_{\PP(V)^-} = (\dim M +1) \zeta_V^-$, it follows that $-K_{\PP(T_M)^-}$ is big and nef as well. 
Vice versa, observe that if $(\PP(V),\PP(T_M)) $ has a weak Fano model, then the vector bundles $V$ and $T_M$ have to be big.
Constructing a weak Fano model gives a straightforward way to check the somewhat elusive property that $Z_M$ is affine:

\begin{proposition} \label{prop:strat} 
Let $M$ be a projective manifold, and let $\PP(T_M) \subset \PP(V)$ be a canonical extension associated to some K\"ahler class on $M$.
Assume that $(\PP(V),\PP(T_M)) $ has a weak Fano model $(\PP(V))^-, \PP(T_M)^-)$.  Let
$$ \psi: \PP(V) \dasharrow \PP(V)^-$$
be the corresponding sequence of antiflips and assume that the restriction $\psi_{\vert Z_M}$ is biholomorphic. Then we have a dichotomy:
\begin{itemize}
\item If $\psi(Z_M) \subsetneq \mathbb P(V)^- \setminus \mathbb P(T_M)^-$, then $Z_M$ is not affine.
\item If $\psi(Z_M) = \mathbb P(V)^- \setminus \mathbb P(T_M)^-$, 
then $Z_M$ is affine. 
\end{itemize}
\end{proposition} 

For the proof note that in the first case, the complex space $\psi(Z_M) \simeq Z_M$ is not holomorphically convex \cite[V, \S 1, Thm.4]{GR04}, while in the second case we can apply \cite[3.14]{HP21}.

With this preparation, the strategy of the proof of Theorem \ref{theorem-V5} is clear:
in Section \ref{section-V5} we will construct by hand a weak Fano model of $(\PP(V),\PP(T_M))$. Thanks to the explicit construction we can verify that we are in the first case of the dichotomy \ref{prop:strat}.

\subsection{Future directions}

Let $M$ be a projective manifold such that a canonical extension $(\PP(V), \PP(T_M))$
admits a weak Fano model $\merom{\psi}{\PP(V)}{\PP(V)^-}$. Then the tangent bundle $T_M$ is big and nef in codimension one, i.e. the tautological bundle $\zeta_M \rightarrow \PP(T_M)$ is nef in codimension one.
If we allow $\psi $ to be the identity, the class  of manifolds such that $T_M$ is big and nef in codimension one (resp. nef in codimension one) contains in particular the manifolds $M$ such that $T_M$ is big and nef 
(resp. simply nef). 

\begin{problem}  \label{problem-codimension-1}
What are the projective manifolds $M$ such that $T_M$ is big and nef in codimension one, but not nef? 
\end{problem} 

By \cite[Cor.1.2]{HL21} there is exactly one Fano threefold $M$ with $\rho(M) = 1$ such that   $T_M$ big and nef in codimension one, but not nef, namely $M = V_5$.  The proof of Theorem \ref{theorem-V5} is quite involved and uses heavily the geometry of the del Pezzo threefold $V_5$. Thus, as a first step towards Problem \ref{problem-codimension-1}, one might wonder whether there is a cohomological proof. 
In this context, it follows from the work of J.Zhang \cite{Zha08} that $Z_M$ is not affine provided 
$$ 
H^{i}(Z_M, \Omega^j_M) \ne 0 
$$
some $i \geq 1, j \geq 0$ (for algebraic cohomology). 
These groups however seem hard to compute. 

Finally note that Conjecture \ref{MainConj} has the following analytic version:
 
\begin{conjecture}\label{ana} 
Let $M$ be a compact K\"ahler manifold such that the canonical extension is Stein for some K\"ahler class. 
Then $T_M$ is nef. 
\end{conjecture} 

The structure of K\"ahler manifolds is well understood up to the case of Fano manifolds, \cite{DPS94}, in which case $M$ is conjectured to be rational homogeneous,
\cite{CP91}.  Thus, Conjecture \ref{ana} includes 

\begin{conjecture} Let $M$ be a Fano manifold. Assume that $Z_M$ is Stein for some K\"ahler class. Then $Z_M$ is affine. 
\end{conjecture}

{\bf Acknowledgements.} We thank Jie Liu for very useful discussions about $V_5$, see in particular Lemma \ref{liu}.  The first-named author thanks the Institut Universitaire de France for providing excellent working conditions.

\section{Notation and first results}

We work over the complex numbers, for general definitions we refer to \cite{Har77, Kau83}. 
Complex spaces and varieties will always be supposed to be irreducible and reduced. 

We use the terminology of \cite{Deb01, KM98}  for birational geometry and notions from the minimal model program. We follow \cite{Laz04} for algebraic notions of positivity.
For positivity notions in the analytic setting, we refer to \cite{Dem12}.

Given a vector bundle $V$ on a manifold $B$, we denote by $\PP(V) \rightarrow B$
the projectivisation in Grothendieck's sense, and by $\zeta_V = \zeta_{\PP(V)}:=\sO_{\PP(V)}(1)$ the tautological line bundle on $\PP(V)$. In the special case of the tangent bundle $T_M$ of a manifold $M$,
we denote the tautological bundle $\zeta_{\PP(T_M)}$ simply by $\zeta_M$.

Given any non-zero class $\alpha \in H^1(M,\Omega_M)$,  usually a K\"ahler class, we let 
$$ 0 \to \sO_M \to V \to T_M \to 0 $$
be the associated nonsplitting extension. We set 
$$ Z_M = \PP(V) \setminus \PP(T_M). $$ 

In this section we investigate the structure of Mori fibre spaces of projective manifolds $M$ such that $Z_M$ is Stein. 

\begin{theorem} \label{thm1a}   Let $M$ be a projective manifold of dimension $n$,  $S \subset M$ be a smooth projective surface. Let $C \subset S$ be a smooth rational curve with $C^2 < 0$ and
 let $\varphi: S \to S'$ be the contraction of $C$. Assume the following 
 \begin{enumerate} 
 \item  the restriction $N_{S/M} \otimes \sO_C $ 
is trivial, or, more generally, that $\varphi_* (N_{S/M}) $ is locally free, 
\item  the canonical morphism $R^1\varphi_*(T_S) \to R^1\varphi_*(T_M \otimes \sO_S)$ is injective. 
\end{enumerate} 
Then $Z_M$ is not Stein for any K\"ahler class.
\end{theorem} 

\begin{remark} \label{remark-thm1a} 
The technical conditions in Theorem \ref{thm1a} can be easily checked in a number of situations:
\begin{itemize}
\item If $C^2 = -1$, then $R^1\varphi_*(T_S) = 0$, so the second condition is trivial.
Thus if $N_{S/M} \otimes \sO_C$ is trivial, Theorem \ref{thm1a} applies. 
\item If $\dim M=3$ and $S \cdot C \leq 0, C^2=-1$, Theorem \ref{thm1a} also applies:
by what precedes this is clear if $S \cdot C=0$. If $S \cdot C<0$, then 
$N_{S/M} \otimes \sO_C \simeq N_{C/S}^{\otimes k}$ with $k \in \N$.
Therefore $N_{S/M} \otimes \sO_U \simeq \sO_U(kC) $ in a neighborhood $U$ of $C$, thus $\varphi_* (N_{S/M})$ is locally free. 
\item  If $\dim M=3$ and $S \cdot C < 0, C^2 \leq -2$, the same arguments applies 
provided $N_{S/M} \otimes \sO_C = N_{C/S}^{\otimes k}$  with $k \in \N$
\item The injectivity of the canonical morphism $R^1\varphi_*(T_S) \to R^1\varphi_*(T_M \otimes \sO_S)$ is satisfied if the tangent bundle sequence 
$$ 0 \to T_S \to T_M \otimes \sO_S \to N_{S/M} \to 0 $$
splits.
\end{itemize}
\end{remark} 

\begin{problem} \label{T1} If $\dim M=3$, generalize Theorem \ref{thm1a}  to the case that  $C$ is a smooth rational curve with $C^2 \leq  -2$ and $S \cdot C \leq 0$. 
\end{problem}

\begin{proof} [Proof of Theorem \ref{thm1a}]
The canonical extension defining $Z_M = \PP(V_M) \setminus \PP(T_M)$ restricts to an exact sequence
\begin{equation} \label{eq0}  0 \to \sO_S \to V_M \otimes \sO_S  \to T_M \otimes \sO_S \to 0 \end{equation} 
Applying $\varphi_*$ and using $R^1\varphi_*(\sO_S) = 0 $ yields an exact sequence
\begin{equation} \label{basic} 0 \to \sO_{S'} \to \varphi_*(V_M \otimes \sO_S) =: V'  \to \varphi_*(T_M \otimes \sO_S) =:T' \to 0.\end{equation} 

We set 
$$ \tilde Z_S := \mathbb P(V_M \otimes \sO_S) \setminus \mathbb P(T_M \otimes \sO_S)$$
and 
$$ \tilde Z_{S'} := \mathbb P(V') \setminus \mathbb P(T').$$ 

\vskip .2cm {\it Claim A.} Sequence (\ref{basic}) satisfies the extension property of \cite[Lemma 4.9]{HP21} near the singular point $s$ of $S'$. This will in turn be a consequence of \cite[Thm.4.5]{HP21}, Part 3. \\
Given this, we conclude as follows. 
Denote by $\pi: \tilde Z_S \to S$ and $\tau:  \tilde Z_{S'} \to S'$ the projections. 
Assume that $Z_M$ is Stein. Then the closed subspace  $\tilde Z_S \subset  Z_M$ is Stein,  hence   $\tilde Z_S \setminus \pi^{-1}(C)$ is Stein as well, since $\pi^{-1}(C)$ is a Cartier divisor \cite[V, \S 1, Thm.1 d)]{GR04}. 
On the other hand, 
$$ \tilde Z_S \setminus \pi^{-1}(C) \simeq \tilde Z_{S'} \setminus \tau^{-1}(s),$$
and $ \tilde Z_{S'} \setminus \tau^{-1}(s)$ is not Stein by \cite[Lemma 4.9]{HP21}, 
a contradiction. 

{\em Proof of Claim A.}
The extension class ${\zeta_M}_{\vert S} \in H^1(S, \Omega^1_M \otimes \sO_S) $ defining our given extension projects onto a class $\zeta_S \in H^1(S,\Omega^1_S) $ defining a canonical extension
$$ 0 \to \sO_S \to V_S \to T_S \to 0$$ 
such that the following diagram commutes

\vspace{-1ex}
\begin{equation} \label{bigdiagram1a}
\xymatrix{
&  &  0 \ar[d]  & 0 \ar[d]  & 
\\
0 \ar[r] & \sO_S \ar[d]^{=} \ar[r] & V_S  \ar[d]  \ar[r] & T_S \ar[d]  \ar[r] & 0
\\
0 \ar[r] & \sO_S \ar[r] & V_M \otimes \sO_S \ar[r] \ar[d] & T_M \otimes \sO_S \ar[r] \ar[d] & 0
\\
 & & N_{S/M} \ar[r]^-{=} \ar[d] &N_{S/M}   \ar[d] &  \\
 & & 0 &0 &
}
\end{equation} 

By assumption, the canonical morphism
$$
R^1\varphi_*(T_S) \to R^1\varphi_*(T_M \otimes \sO_S)
$$ 
is injective.
Thus pushing forward the whole diagram to $S'$ we obtain a commutative diagram of exact sequences
\begin{equation} \label{bigdiagram2a}
\xymatrix{
&  &  0 \ar[d]  & 0 \ar[d]  & 
\\
0 \ar[r] & \sO_{S'} \ar[d]^{=} \ar[r] & \varphi_* V_S  \ar[d]  \ar[r] & \varphi_* T_S \ar[d]  \ar[r] & 0
\\
0 \ar[r] & \sO_{S'} \ar[r] & \varphi_* (V_M \otimes \sO_S) \ar[r] \ar[d] & \varphi_* (T_M \otimes \sO_S) \ar[r] \ar[d] & 0
\\
 & & \varphi_* N_{S/M} \ar[r]^-{=} \ar[d] & \varphi_* N_{S/M}   \ar[d] &  \\
 & & 0  & 0&
}
\end{equation} 
Taking biduals this defines a commutative diagram of reflexive sheaves 
\begin{equation} \label{bigdiagram3a}
\xymatrix{
&  &  0 \ar[d]  & 0 \ar[d]  & 
\\
0 \ar[r] & \sO_{S'} \ar[d]^{=} \ar[r] & (\varphi_* V_S)^{**}  \ar[d]  \ar[r] & (\varphi_* T_S)^{**} \ar[d]  \ar[r] & 0
\\
0 \ar[r] & \sO_{S'} \ar[r] & (\varphi_* (V_M \otimes \sO_S))^{**} \ar[r] \ar[d] & (\varphi_* (T_M \otimes \sO_S))^{**} \ar[r] \ar[d] & 0
\\
 & &\varphi_*N_{S/M} =  (\varphi_* N_{S/M})^{**} \ar[r]^-{=} \ar[d] & (\varphi_* N_{S/M})^{**} = \varphi_*(N_{S/M}) \ar[d] &  \\
 & & 0 & 0 &
}
\end{equation}

\vskip .2cm 
{\it Claim B.} The sequences in this diagram are also exact.

{\em Proof of Claim B.} 
We start by observing that the first row is exact:  the class $\zeta_S$ determines canonically a class $\zeta'$ on $S'$, see \cite[4.7]{HP21}, thereby defining an extension of $T_{S'}$ by $\sO_{S'}$. Since $(\varphi_* T_S)^{**} \simeq T_{S'}$, the first line coincides
in codimension one with this extension, so they coincide.

Note also that the statement is clear in the complement of the point $s$, so we can replace $S'$
by a Stein neighbourhood $U'$ of the point $s$.
Recall that $\varphi_*N_{S/M}$ is locally  free.
Since the columns of \eqref{bigdiagram2a} are exact, they split on the Stein neighbourhood $U'$, i.e.
$$
\varphi_* (V_M \otimes \sO_S) \otimes \sO_{U'}
\simeq (\varphi_* V_S \otimes \sO_{U'}) \oplus \sO_{U'}^{\dim M-2}
$$
and analogously for $\varphi_* (T_M \otimes \sO_S) \otimes \sO_{U'}$. In particular taking biduals does not change the exactness of the columns. Thus we are left to show the exactness of the second row, but this follows from a diagram chase. This shows Claim B. 

The first row of Diagram (\ref{bigdiagram3a}) reads
$$ 0 \to \sO_{S'} \to V_{S'} \to T_{S'} \to 0,$$
hence has the extension property by \cite[Thm.4.5]{HP21}, proof of Part 3.
Thus the middle row of  Diagram (\ref{bigdiagram3a}) has the extension property as well. 
But restricted to $S' \setminus \{s\}$, this sequence is just Sequence (\ref{basic}). 
Hence Sequence (\ref{basic}) has the extension property, too, which settles Claim A and finishes the proof. 
\end{proof}

\begin{corollary} \label{cor1} 
Let $M$ be a projective manifold such that $Z_M$ is Stein for some K\"ahler class. 
Let $f: M \to N$ be a fibre space and $\dim N = \dim M - 2$. 
Then any smooth fibre $S$ is a minimal surface, i.e., does not contain any $(-1)$-curve. 

In particular, if $f$ is a Mori fibre space, then 
any smooth fibre of $f$ is either $\mathbb P^2$ or $\PP^1 \times \PP^1$. 
\end{corollary} 

\begin{proof} Since $S$ is a smooth fibre of the fibration $f$, the normal bundle $N_{S/M}$ is trivial. Thus Theorem \ref{thm1a} applies to any $(-1)$-curve in $S$, cf. Remark
\ref{remark-thm1a}. 
\end{proof}

\begin{corollary} \label{cor2} 
Let $M$ be a smooth projective threefold such that $Z_M$ is Stein for some K\"ahler class. Let $f: M \to N$ be an elementary Mori contraction to a surface $N$. 
Then $f$ is a $\PP^1$-bundle. 
\end{corollary} 

\begin{proof}
By \cite{Mor82}, the fibration $f$ is either a $\mathbb P^1$-bundle or a conic bundle with singular fibres, with singular fibres in codimension one being a line pair. 
Assume the latter and let $H \subset N$ be a general hyperplane section. Then $S = \varphi^{-1}(H)$ is a smooth surface, and $f_{\vert S}: S \to H$ contains line pairs as singular 
fibres. Choose a singular fibre $f^{-1}(x_0)$ and let $C \subset f^{-1}(x_0)$ be an irreducible component. Then $C \subset S$ is a $(-1)$-curve and $S \cdot C = 0$. 
But then by Theorem \ref{thm1a}, $Z_M$ cannot be Stein. Hence $f$ must be a $\mathbb P^1$-bundle.
\end{proof}

\section{The del Pezzo threefold of degree five} \label{section-V5}

We fix the following setup for the whole section:
we denote by 
$$
M := \mbox{Gr}(2,5) \cap \PP^6 \subset \PP^9
$$ 
the del Pezzo threefold of degree five, often quoted as $V_5$. 
We let 
$$Y = \PP(T_M)$$
with projection $\pi: \PP(T_M) \to M$ and tautological line bundle $\zeta_M = \sO_{\PP(T_M)}(1)$. 
We denote by $V$ the canonical extension and set
$$ X = \PP(V),$$
so that $Z_M = X \setminus Y$. 
Recall that $\zeta_V$ denotes the tautological line bundle on $X$ and that $Y \in \vert \zeta_V \vert$. 

Our goal is to prove

\begin{theorem} \label{theorem-V5-details}
Let  $M$ be the Fano threefold $V_5$. Then $Z_M$ is not Stein. 
More precisely, there exists an antiflip 
$$\psi:  \PP(V) \dasharrow \PP(V)^- $$ with the following properties
\begin{enumerate}
\item $\psi$ 
induces a biholomorphic map 
$$ Z_M \to \PP(V)^- \setminus (\PP(T_M)^- \cup A) $$
where $\PP(T_M)^- $ is the strict transform of $\PP(Z_M)$  and $A$ is (non-empty) analytic set of codimension two, which is not contained in $\PP(T_M)^-$;
\item  $\PP(T_M)^-$ is an antiflip of $\PP(T_M)$;
\item $\PP(V)^-$ and $\PP(T_M)^-$ are weak $\mathbb Q$-Fano;
\item  $\PP(V)^- \setminus \PP(T_M)^-$ is affine.
\end{enumerate}

\end{theorem} 

In other words, $(\PP(V),\PP(T_M)) $ has the weak Fano model $(\PP(V)^-,\PP(T_M)^-)$.

\subsection{Geometry of  $\PP(T_M)$}

We denote by $\sO_M(1)$ the restriction of the tautological bundle to $M$. It is well-known \cite[Prop.3.2.4]{EncV} that $\sO_M(1)$ is the ample generator of $\pic(M)$. 

The threefold $M$ has been studied by many authors, e.g., by \cite{MU83, FN89, San14};
here we mainly follow the exposition in \cite[Ch.7]{CS16}. The manifold $M$ is almost homogeneous 
with automorphism group $\mbox{Aut}(M) \simeq \mbox{PSL}_2 \C$ \cite[Prop.7.1.10]{CS16}
and has three orbits \cite[Thm.7.1.4, Thm.7.1.9]{CS16}: 
\begin{itemize}
\item The unique closed orbit is 
a smooth rational curve $B$ which is a rational normal curve of degree six in
$M \subset \PP^6$. We have \cite[Lemma 7.2.8]{CS16}
\begin{equation} \label{normal-B}
N_{B/M} \simeq \sO_B(5)^{\oplus 2}.
\end{equation}
\item There is a two-dimensional orbit and we denote its closure by $\bar S$: it is obtained as the locus of lines $l \subset M$ that are tangent lines of $B$. 
By \cite[Cor.1.2]{FN89} these are exactly the lines $l \subset M$ such that
\begin{equation} \label{negative-splitting}
T_M \otimes \sO_l \simeq \sO_l(2) \oplus \sO_l(1) \oplus \sO_l(-1).
\end{equation}
The divisor $\bar S$ is an element of the anticanonical system $|\sO_M(2)|$.
\item The three-dimensional orbit is $M \setminus \bar S$.
\end{itemize}
The threefold $M \subset \PP^6$ is covered by lines, and its Fano surface
of lines is isomorphic to $\PP^2$ \cite[Thm.I]{FN89}, \cite[Rem.2.28]{San14}.
We denote by
\begin{equation} \label{universal-lines}
\holom{q}{\mathcal U}{\PP^2}
\end{equation}
the universal family of lines on $M$ and by $\holom{e}{\mathcal U}{M}$ the evaluation morphism. By \cite[Thm.7.1.9]{CS16}, \cite[Lemma 2.3]{FN89} the evaluation morphism has degree three
and is \'etale over $M \setminus \bar S$.

The lines of splitting type \eqref{negative-splitting} are parametrised by a smooth conic
$C \subset \PP^2$ \cite[Thm.7.1.9]{CS16}
and we denote by 
\begin{equation} \label{universal-negative-lines}
\holom{q|_{\mathcal U_C}}{\mathcal U_C}{C}
\end{equation}
the restriction of the universal family.
By \cite[Thm.7.1.9(iii)]{CS16} the lines contained in $\bar S$ are disjoint, 
so the evalution map 
$$
\holom{e_C}{\mathcal U_C}{\bar S}
$$
is injective. Thus we can identify $e_C$ with the normalisation of $\bar S$.
The singular locus of $\bar S$ is the curve $B$ \cite[Lemma 7.2.2(i)]{CS16},
so we see that $e_C$ is an isomorphism on $\mathcal U_C \setminus \tilde B$
where by $\tilde B \subset \mathcal U_C$ we denote the set-theoretical pre-image of the curve $B \subset \bar S$.
The surface $\mathcal U_C$ is isomorphic to a quadric surface $\PP^1 \times \PP^1$ \cite[Lemma 7.2.2(vi)]{CS16}. In order to keep the notation transparent we will write
\begin{equation} \label{notation-product}
\mathcal U_C \simeq C \times l,
\end{equation}
and identify the morphism $q|_{\mathcal U_C}$ with the projection $\holom{p_C}{C \times l}{C}$.
We also know \cite[Lemma 7.2.3(iii)]{CS16} that
\begin{equation} \label{classB}
\tilde B \in  \vert \sO_{C \times l}(1,1) \vert
\end{equation}
Since $K_{\bar S} \simeq \sO_{\bar S}$ by adjunction and $K_{C \times l} \simeq
\sO_{C \times l}(-2,-2)$, the map $e_C$ ramifies with multiplicity two
along $B$. Finally note that by \cite[Lemma 7.2.3]{CS16} we have
\begin{equation} \label{classOMone}
e_C^* \sO_M(1) \simeq \sO_{C \times l}(5,1).
\end{equation}

Observe that the line bundle $\zeta_M \otimes \pi^* \sO_M(1)$ is nef and globally generated \cite[Prop.6.1]{FL22}: indeed $\Omega^2_{\PP^9}(3)$ is globally generated, and we have a surjection
$$
\Omega^2_{\PP^9}(3) \otimes \sO_M \twoheadrightarrow \Omega^2_M(3) \simeq T_M(1).
$$
Denote by
\begin{equation} \label{define-varphiM}
\holom{\varphi_M}{\PP(T_M)}{\PP^N}
\end{equation}
the morphism defined by the global sections of $\zeta_M \otimes \pi^* \sO_M(1)$.

Let $l \subset M$ be a line parametrised by $C \subset \PP^2$, i.e., $l$ has splitting type 
\eqref{negative-splitting}. Then we denote by 
$$
\tilde l \subset \PP(T_M)
$$
the lifting given by the quotient line bundle
$T_M \otimes \sO_l \twoheadrightarrow \sO_l(-1)$.
Note that by construction
$$
c_1(\zeta_M \otimes \pi^* \sO_M(1)) \cdot \tilde l = 0,
$$
so the curve $\tilde l$ is contracted by $\varphi_M$ onto a point.
Let $$S \subset \PP(T_M)$$ be the union of the curves $\tilde l$.

We can get a better description of $S$ by looking at the
vector bundle $e_C^* T_M \otimes \sO_{\mathcal U_C}$:
the surface $\mathcal U_C$
is the ramification divisor of $e$, so the tangent map $T_e$ has rank two in the generic point of $\mathcal U_C$ and the image of $T_e|_{\mathcal U_C}$ is given by the image of 
$$
T_{\mathcal U_C} \rightarrow e^* T_M \otimes \sO_{\mathcal U_C}.
$$
Since $e_C$ is not immersive along the  curve $\tilde B \subset \mathcal U_C$ and  since $\tilde B$ is embedded into $M$ via $e$,
the map $T_{\mathcal U_C} \rightarrow e^* T_M$ has rank one along $\tilde B$. 
Thus the image of the injective map $T_{\mathcal U_C} \rightarrow e^* T_M$ is not saturated,
and we denote by
$$
K \subset e^* T_M \otimes \sO_{\mathcal U_C}
$$
the saturation. Since $T_{\mathcal U_C} \simeq T_{C \times l} \simeq T_{\mathcal U_C/C} \oplus p_C^* T_C$ and the evaluation map $e$
is immersive on the fibres of the universal family (the images are lines), the map
$$
T_{\mathcal U_C/C} \rightarrow e^* T_M \otimes \sO_{\mathcal U_C}
$$
has rank one in every point. Therefore we have an extension
$$
0 \rightarrow T_{\mathcal U_C/C} \rightarrow K \rightarrow L \rightarrow 0
$$
with $L$ a reflexive sheaf of rank one and a non-zero morphism $p_C^* T_C \rightarrow L$. 
Since $p_C^* T_C \rightarrow L$ vanishes along $B$ 
and $\sO_{\mathcal U_C}(B) \simeq \sO_{C \times l}(1,1)$ we have $L \simeq \sO_{C \times l}(2,0) \otimes \sO_{C \times l}(m,m)$ with $m$ the vanishing order. We claim that $m=1$:
restricting to a line $l \subset \mathcal U_C$ we have
$$
T_{\mathcal U_C} \otimes \sO_l \simeq \sO_l(2) \oplus \sO_l \rightarrow e^* T_M \otimes \sO_l \simeq 
\sO_l(2) \oplus \sO_l(1) \oplus \sO_l(-1).
$$
Since $K$ is the saturation of $T_{\mathcal U_C}$ we obtain that
$K \otimes \sO_l \simeq \sO_l(2) \oplus \sO_l(1)$.
Thus we have $m=1$.

The following statement summarises our construction:

\begin{lemma}  \label{lemma-zetaM}
The evaluation map $\holom{e_C}{\mathcal U_C}{\bar S \subset M}$
factors through a morphism $\holom{\tilde e_C}{\mathcal U_C}{S \subset \PP(T_M)}$ such that for a line $l \subset \bar S$, the set-theoretical preimage
$\tilde l \subset S$ is given by the quotient line bundle
$$
T_M \otimes \sO_l \twoheadrightarrow \sO_l(-1).
$$
Moreover we have $\tilde e_C^* \zeta_M \simeq \sO_{C \times l}(7,-1)$.
\end{lemma}

\begin{remark*}
At this point we do not know
whether $\mathcal U_C \rightarrow S$ is an isomorphism, this will be shown in Lemma \ref{lemma-exceptional-locus-varphiM}.
\end{remark*}

\begin{proof}
By what precedes we have an extension
\begin{equation} \label{tangent-over-UC}
0 \rightarrow K \rightarrow e_C^* T_M \otimes \sO_{\mathcal U_C} \rightarrow Q \rightarrow 0
\end{equation}
where $K \subset e_C^* T_M \otimes \sO_{\mathcal U_C}$ is a rank two subbundle given by an extension
\begin{equation} \label{extension-K}
0 \rightarrow \sO_{C \times l}(0,2) \rightarrow K \rightarrow \sO_{C \times l}(3,1) \rightarrow 0.
\end{equation} 
Since by \eqref{classOMone} one has $e_C^* \omega_M^*= e_C^* \sO_M(2H) = \sO_{C \times l}(10,2)$, we obtain 
that $Q \simeq \sO_{C \times l}(7,-1)$. Let $\holom{\tilde e_C}{\mathcal U_C}{S \subset \PP(T_M)}$ be the map determined by the quotient $e_C^* T_M \rightarrow Q$.
Then $\tilde e_C$ factors $e_C$ and by the universal property of the tautological bundle one has $\tilde e_C^* \zeta_M \simeq \sO_{C \times l}(7,-1)$.
Note also that the restriction of $Q$ to a line $l \subset \mathcal U_C$
is $\sO_l(-1)$. Since $l$ identifies to its image in $\bar S$, we see that
$\tilde e_C(l)$ corresponds to the lifting $T_M \otimes \sO_l \twoheadrightarrow \sO_l(-1)$.
\end{proof}

\begin{lemma}(Jie Liu) \label{liu}
The morphism $\varphi_M$ contracts exactly the curves $\tilde l$ onto points.
\end{lemma}

This statement was communicated to us by Jie Liu, his proof uses the blow-up construction from \cite[Sect.7.2]{CS16}. For our purpose it is more convenient to use the description of $e^* T_M \otimes \sO_{\mathcal U_C}$.

\begin{proof}
Let $\tilde D \subset \PP(T_M)$ be a curve contracted by $\varphi_M$ and
set $D:=\pi(\tilde D)$. 

Assume first that $D \not\subset \bar S$. Since $T_M$ is globally generated in the complement of $\bar S$, we have $c_1(\zeta_M) \cdot \tilde D \geq 0$ and
hence $c_1(\zeta_M \otimes \pi^* \sO_M(1)) \cdot \tilde D>0$, a contradiction.

Thus we may assume that $D \subset \bar S$. By \eqref{tangent-over-UC}
we have an extension
$$
0 \rightarrow K \otimes e_C^* \sO_M(1)
\rightarrow e_C^* (T_M \otimes \sO_M(1)) \rightarrow
Q \otimes e_C^* \sO_M(1)
\rightarrow 0.
$$
Since $e_C$ is finite and $K$ is nef by \eqref{extension-K}, 
the vector bundle $K \otimes e_C^* \sO_M(1)$ is ample. Thus the non-ample locus of $\zeta_M \otimes \pi^* \sO_M(1)$ is the image 
of $\PP(Q \otimes e_C^* \sO_M(1)) \subset \PP(e_C^* (T_M \otimes \sO_M(1)))$.
By construction this is exactly the surface $S$, so we have $D \subset S$. By Lemma \ref{lemma-zetaM} we have
$$
Q \otimes e_C^* \sO_M(1) \simeq \sO_{C \times l}(7,-1) \otimes \sO_{C \times l}(5,1) \simeq
\sO_{C \times l}(12,0),
$$
so $\zeta_M \otimes \pi^* \sO_M(1)$ restricted to $S$ is nef, not big and has
degree zero exactly on the curves $\tilde l$.
\end{proof}

\begin{lemma} \label{lemma-exceptional-locus-varphiM}
We have 
$$
\mbox{\rm Exc}(\varphi_M) = S \simeq \PP^1 \times \PP^1,
$$ i.e. the surface $S$ is smooth and $\holom{\tilde e_C}{\mathcal U_C}{S}$
is an isomorphism. 
\end{lemma}

\begin{proof}
By Lemma \ref{liu}, the first statement is clear. Since $\varphi_M$ contracts no other curves,
the image $\varphi_M(S)$ has dimension one, so $S \rightarrow \varphi_M(S)$
is a fibration with set-theoretical fibres the curves $\tilde l$.
Since $\mathcal U_C \rightarrow \bar S$ is bijective, the normalisation of $S$ is isomorphic to $\PP^1 \times \PP^1$, so
$$
\PP^1 \times \PP^1 \rightarrow S \rightarrow \varphi_M(S)
$$
identifies to one of the projections. Thus $\varphi_M(S)$ is a smooth rational curve: in fact, ${\rm PSL_2}\mathbb C$ acts on the image of $\varphi_M$ and $\varphi_M(S)$ is 
an orbit as image of the $1$-dimensional orbit $B$. 
We also know that $S$ is smooth in the complement of the preimage of $B \subset \bar S$ and this preimage maps surjectively
onto $\varphi_M(S)$, so every fibre of $S \rightarrow \varphi_M(S)$ meets the smooth locus
of $S$.

We now argue by contradiction: if $p \in S$ is a singular point, the fibre $\varphi_M^* \varphi_M(p)$ is a Cartier divisor meeting a singular point, hence is singular. Yet $\varphi_M^* \varphi_M(p)$
is isomorphic to $\tilde l$ in its general point. Thus the Cartier divisor $\varphi_M^* \varphi_M(p)$ is reduced and coincides with the set-theoretical fibre $\tilde l$.
Thus $\varphi_M^* \varphi_M(p)$ is smooth, a contradiction.
\end{proof}

\begin{remark*} 
In order to keep the notation transparent, we will write
$$
S \simeq C \times \tilde l,
$$
and identify $\varphi_M|_S$ with the projection $\holom{p_C}{C \times \tilde l}{C}$.
\end{remark*}

The base locus $Bs\vert \zeta_M \vert$ of $\vert \zeta_M \vert $ is easy to compute:

\begin{lemma} \label{lem:base} 
$Bs\vert \zeta_M\vert = S\cup \mathbb P(N_{B/M}).$ 
\end{lemma}

\begin{proof} From the description of the orbits of $\mbox{Aut}(M)$ at the start of this subsection it is clear that  $Bs\vert \zeta_M \vert $ is contained in $\pi^{-1}(\overline S)$.
Clearly $S \subset Bs\vert \zeta_M \vert $, since the restriction of $\zeta_M$
to $S$ is not pseudoeffective.
Let $x \in \overline S \setminus B$. Then 
$$ {\rm Ker} (H^0(M,T_M) \to T_{M,x}) $$
has dimension one, and so has
$$ 
{\rm Ker}\left(
H^0(\PP(T_M), \zeta_M) \to H^0(\pi^{-1}(x), \zeta_M)
\right).
$$ 
Thus $(Bs\vert \zeta_M \vert)\cap \pi^{-1}(x) $ is a single point, i.e., an irreducible
component $S' \neq S$ of $Bs\vert \zeta_M\vert$ is mapped into $B$. 

Finally, since the image of the restriction map
$$ 
H^0(M,T_M) \to H^0(B, {T_M} \otimes \sO_B) 
$$
is equal to $H^0(B,T_B)$, it follows that $\mathbb P(N_{B/M}) \subset Bs\vert \zeta_M \vert$ and that $\pi^{-1}(B) \not \subset Bs\vert \zeta_M \vert $.
\end{proof} 

Recall that $\PP(T_M)=Y$ has a contact structure \cite{Le95}: 
there exists a corank one subbundle $\sF \subset T_Y$ that fits into an exact sequence
\begin{equation} \label{contact1}
0 \rightarrow \sF \rightarrow T_Y \rightarrow \zeta_M \rightarrow 0.
\end{equation}
Moreover the Lie-bracket induces a morphism
$$
\sF \otimes \sF \rightarrow \zeta_M
$$
that is every non-degenerate, so we have an isomorphism 
\begin{equation} \label{contact2}
\sF \simeq \sF^* \otimes \zeta_M.
\end{equation}
 We also have an exact sequence
\begin{equation} \label{contactA}
0 \rightarrow T_{Y/M} \rightarrow \sF \rightarrow \Omega_{Y/M} \otimes \zeta_M \rightarrow 0.
\end{equation} 
We will now determine the normal bundle of a curve $\tilde l \subset Y=\PP(T_M)$:

\begin{lemma} \label{lemma-splitting-type-Y}
We have
$$
T_Y \otimes \sO_{\tilde l} \simeq \sO_{\tilde l}(2) \oplus \sO_{\tilde l} \oplus \sO_{\tilde l}(-1) \oplus \sO_{\tilde l}(-2)^{\oplus 2}
$$
\end{lemma}

\begin{proof} 
The curve $\tilde l$ is the section of $\PP(T_M \otimes \sO_l) \rightarrow l$ corresponding to
the negative quotient  $\sO_{l}(2) \oplus \sO_{l}(1) \oplus \sO_{l}(-1) \rightarrow \sO_{l}(-1)$.
Thus we have
$$
T_{Y/M} \otimes \sO_{\tilde l} \simeq T_{\PP(T_M \otimes \sO_l)/l} \otimes \sO_{\tilde l}  \simeq N_{\tilde l/\PP(T_M \otimes \sO_l)} \simeq \sO_{\tilde l}(-3) \oplus \sO_{\tilde l}(-2).
$$
Since $\zeta_M \cdot \tilde l=-1$ this implies 
that 
$$
\Omega_{Y/M} \otimes \zeta_M \otimes \sO_{\tilde l} \simeq \sO_{\tilde l}(2) \oplus \sO_{\tilde l}(1).
$$
Sequence (\ref{contactA}) therefore reads
$$
0 \rightarrow \sO_{\tilde l}(-3) \oplus \sO_{\tilde l}(-2) \rightarrow \sF \otimes \sO_{\tilde l}
\rightarrow \sO_{\tilde l}(2) \oplus \sO_{\tilde l}(1) \rightarrow 0
$$
Thus if we write
$$
\sF \otimes \sO_{\tilde l} \simeq \oplus_{i=1}^4 \sO_{\tilde l}(a_i)
$$
with $a_1 \geq a_2 \geq a_3 \geq a_4$, then $a_1 \leq 2$. 

Consider now the natural inclusion
$T_{\tilde l} \hookrightarrow T_Y \otimes \sO_{\tilde l}$: since $\zeta_M \cdot \tilde l=-1$ we deduce
from \eqref{contact1} that $T_{\tilde l} \hookrightarrow \sF \otimes \sO_{\tilde l}$. Hence we have $a_1=2$. Since by \eqref{contact2} we have $\sF \otimes \sO_{\tilde l}  \simeq \sF^* \otimes \sO_{\tilde l}(-1)$ this also implies that $a_4=-3$. 
Since $c_1(\sF) \cdot \tilde l=-2$, we are left with two possibilities:
$$
\sF \otimes \sO_{\tilde l} \simeq \sO_{\tilde l}(2) \oplus \sO_{\tilde l} \oplus \sO_{\tilde l}(-1) \oplus \sO_{\tilde l}(-3)
$$ 
or
$$
\sF \otimes \sO_{\tilde l} \simeq \sO_{\tilde l}(2) \oplus \sO_{\tilde l}(1) \oplus \sO_{\tilde l}(-2) \oplus \sO_{\tilde l}(-3).
$$ 
Let us exclude the latter case: since $S \subset Y$ is a smooth surface, we have
an injective morphism of vector bundles
$$
\sO_{\tilde l}(2) \oplus \sO_{\tilde l} \simeq T_S \otimes \sO_{\tilde l} \hookrightarrow T_Y \otimes \sO_{\tilde l}. 
$$
Using again that $\zeta_M \cdot \tilde l=-1$ and \eqref{contact1}, the inclusion lifts
to an injective morphism of vector bundles
$$
\sO_{\tilde l}(2) \oplus \sO_{\tilde l} \hookrightarrow \sF \otimes \sO_{\tilde l}. 
$$
Yet if  $\sF \otimes \sO_{\tilde l} \simeq \sO_{\tilde l}(2) \oplus \sO_{\tilde l}(1) \oplus \sO_{\tilde l}(-2) \oplus \sO_{\tilde l}(-3)$,
such an injective morphism does not exist.

Using  again the extension \eqref{contact1} we deduce
$$
T_Y \otimes \sO_{\tilde l} \simeq \sO_{\tilde l}(2) \oplus \sO_{\tilde l} \oplus \sO_{\tilde l}(-1)^{\oplus 2} \oplus \sO_{\tilde l}(-3)
$$
or
$$
T_Y \otimes \sO_{\tilde l} \simeq \sO_{\tilde l}(2) \oplus \sO_{\tilde l} \oplus \sO_{\tilde l}(-1) \oplus \sO_{\tilde l}(-2)^{\oplus 2}.
$$
Thus, it remains to exclude the first case which is exactly the case when the restricted contact sequence
$$ 
0 \to \sF \otimes \sO_{\tilde l}  \to  T_Y \otimes \sO_{\tilde l} \to \zeta_M
\otimes \sO_{\tilde l} \to 0 
$$
splits. By \cite[Prop.2.5]{Le95}, the extension class $a$ of this sequence is the image of $$\frac {2\pi i }{2} [-K_Y]|_{\tilde l} \in H^1(\tilde l ,\Omega_Y \otimes \sO_{\tilde l})$$ 
under a canonical map 
$$
\mu:  H^1(\tilde l,\Omega_Y \otimes \sO_{\tilde l}) \to H^1(\tilde l, \sF \otimes {\zeta_M^*} \otimes \sO_{\tilde l}).
$$
Moreover, the kernel of $\mu$ is contained in $H^1(\tilde l, {\zeta_M^*} \otimes \sO_{\tilde l}) = 0$, so $\mu $ is injective. 
Further, $-K_Y \cdot \tilde l < 0$, hence $a \ne 0$. 
\end{proof} 

Note that the proof above shows that the inclusion $T_S \hookrightarrow T_Y \otimes \sO_S$ factors through an inclusion
$$
T_S \hookrightarrow \sF \otimes \sO_S.
$$ inducing an exact sequence
\begin{equation} \label{contactB}
0 \to T_S \rightarrow \sF \otimes \sO_S \rightarrow \Omega_S \otimes \zeta_M \to 0.
\end{equation} 

Sequences (\ref{contact1}) and (\ref{contactB}) now yield an exact sequence 
\begin{equation} \label{contactC}
0 \rightarrow \Omega_S  \otimes \zeta_M \rightarrow N_{S/Y} \rightarrow \zeta_M \otimes \sO_S \to 0.
\end{equation} 

\begin{proposition} \label{prop:N2}
We have an exact sequence 
$$ 0 \rightarrow \sO_{C \times \tilde l}(-7,2)^{\oplus 2} \rightarrow N^*_{S/Y} \rightarrow \sO_{C \times \tilde l}(-5,1) \rightarrow 0. $$
\end{proposition} 

\begin{proof} By Lemma \ref{lemma-zetaM} we have
 $$
 \zeta_M \otimes \sO_S \simeq \sO_{C \times \tilde l}(7,-1).
 $$
 Since $T_S \simeq \sO_{C \times \tilde l}(2,0) \oplus \sO_{C \times \tilde l}(0,2)$, we deduce from \eqref{contactC} an exact sequence
$$
0 \rightarrow \sO_{C \times \tilde l}(-7,1) \rightarrow N_{S/Y}^* \rightarrow \sO_{C \times \tilde l}(-5,1) \oplus \sO_{C \times \tilde l}(-7,3) \rightarrow 0.
$$
By Lemma \ref{lemma-splitting-type-Y} we have 
$N^*_{S/Y} \otimes \sO_{\tilde l} \simeq  \sO_{\tilde l}(1) \oplus \sO_{\tilde l}(2)^{\oplus 2}$.
Thus if $G$ denotes the kernel of $N^*_{S/Y} \rightarrow \sO_{C \times \tilde l}(-5,1)$,  
it sits in a non-split
extension
$$
0 \rightarrow \sO_{C \times \tilde l}(-7,1) \rightarrow G \rightarrow  \sO_{C \times \tilde l}(-7,3) \rightarrow 0.
$$
Denoting by $p_C$ the projection  onto $C$,  the bundle
$(p_C)_* (G \otimes \sO_{C \times \tilde l}(0,-2))$ has rank two, and the evaluation morphism
$$
p_C^* (p_C)_* (G \otimes \sO_{C \times \tilde l}(0,-2)) \rightarrow G \otimes \sO_{C \times \tilde l}(0,-2)
$$
is surjective. It is straightforward to see that 
$$
(p_C)_* (G \otimes \sO_{C \times \tilde l}(0,-2)) \simeq (p_C)_* \sO_{C \times \tilde l}(-7,1) \simeq \sO_C(-7)^{\oplus 2}.
$$
Thus we have $G \simeq \sO_{C \times \tilde l}(-7,2)^{\oplus 2}$ which proves our claim. 
\end{proof}

\subsection{Geometry of $\PP(V)$ and construction of an antiflip} \label{subsectionPV}

Let us recall that the vector bundle $V$ is given by an extension
$$
0 \rightarrow \sO_M \rightarrow V \rightarrow T_M \rightarrow 0
$$
associated to the K\"ahler class $c_1(\sO_M(1)) \in H^1(M, \Omega_M)$.
Let $\tilde \pi: X=\PP(V) \to M$ denote the projectivisation; thus $Y=\PP(T_M) \subset X$, and  we have an exact
sequence
\begin{equation} \label{tangentXY}
0 \rightarrow T_Y \rightarrow T_X \otimes \sO_Y \rightarrow 
N_{Y/X} \simeq \sO_{\PP(V)}(1) \otimes \sO_Y \simeq \zeta_M \rightarrow 0.
\end{equation}
Let 
$$ \varphi_V: \PP(V) \to \PP_M$$
be the morphism associated with the base point free linear system $\vert \zeta_V \otimes \tilde \pi^* \sO_M(H) \vert $. 

\begin{proposition}  \label{proposition-properties-X}
We have the following properties:
\begin{enumerate}
\item ${\varphi_V }_{\vert Y} = \varphi_M$;
\item the exceptional locus of $\varphi_V$ is $S$, and $\varphi_V$ contracts exactly the curves $\tilde l$;
\item $T_X \otimes \sO_{\tilde l} = \sO_{\tilde l}(2) \oplus \sO \oplus \sO_{\tilde l}(-1)^{\oplus 2 } \oplus \sO_{\tilde l}(-2)^{\oplus 2}$.
\item $N_{S/X} \otimes \sO_{\tilde l} \simeq \sO_{\tilde l}(-1)^{\oplus 2} \oplus \sO_{\tilde l}(-2)^{\oplus 2}$. 
\end{enumerate}
\end{proposition} 

\begin{proof}
The exact sequence
$$
0 \rightarrow \sO_M(1) \rightarrow V \otimes \sO_M(1)  \rightarrow T_M \otimes \sO_M(1) \rightarrow 0
$$
and Kodaira vanishing shows that $H^0(M, V \otimes \sO_M(1))
\rightarrow H^0(M, T_M \otimes \sO_M(1))$ is surjective, and the first statement follows. The second statement is then immediate from Lemma \ref{lemma-exceptional-locus-varphiM}.

The third and fourth statement follow from Lemma \ref{lemma-splitting-type-Y}
and the exact sequence \eqref{tangentXY}.
\end{proof}

\begin{proposition} \label{prop:N3}
We have an exact sequence 
$$
0 \rightarrow \sO_{C \times \tilde l}(-7,2)^{\oplus 2} \rightarrow N^*_{S/X} \rightarrow T \rightarrow 0
$$
where $T$ is a rank two vector bundle given by an extension
$$
0 \rightarrow \sO_{C \times \tilde l}(-7,1) \rightarrow T \rightarrow \sO_{C \times \tilde l}(-5,1) \rightarrow 0
$$
\end{proposition} 

\begin{proof}
By Proposition \ref{prop:N2} there is an exact sequence
$$ 
0 \rightarrow \sO_{C \times \tilde l}(-7,2)^{\oplus 2} \rightarrow N^*_{S/Y} \rightarrow \sO_{C \times \tilde l}(-5,1) \rightarrow 0. 
$$
Consider the exact sequence
$$
0 \rightarrow N_{Y/X}^* \otimes \sO_S \simeq \zeta_M^* \otimes \sO_S \simeq \sO_{C \times \tilde l}(-7,1) \rightarrow
N^*_{S/X} \rightarrow N^*_{S/Y} \rightarrow 0.
$$
By Proposition \ref{proposition-properties-X},d) the direct image sheaf
$(p_C)_* (N_{S/X}^* \otimes \sO_{C \times \tilde l}(0,-2))$ has rank two.
Moreover
we can repeat the argument from the proof of Proposition \ref{prop:N2} to establish a subbundle $\sO_{C \times \tilde l}(-7,2)^{\oplus 2} \rightarrow N^*_{S/X}$. The 
description of $T$ is now straightforward.
\end{proof}

We will now start the construction of the antiflip $X \dashrightarrow X^-$.
 
{\it Step 1.} Let $$\mu_1: X_1 \to X$$
be the blow-up of $X_1$ along $S$, and denote by $Y_1$ the strict transform of $Y$. 
Denote by 
$$
E_1 \simeq \PP(N_{S/X}^*) \rightarrow S
$$ 
the exceptional divisor, and by
$$
F_1 \simeq \PP(N^*_{S/Y})
$$
the exceptional locus of the induced blow-up $Y_1 \to Y$.
By Proposition \ref{prop:N3} and Proposition \ref{prop:N2} there exists
a surface 
\begin{equation} \label{surf} 
S_1 = \PP(\sO_{C \times \tilde l}(-5,1)) \subset F_1 \subset E_1
\end{equation} 
corresponding to a quotient
$N_{S/X}^* \twoheadrightarrow \sO_{C \times \tilde l}(-5,1).$

\begin{lemma} \label{lemma-normal-S1}
We have 
$$
N^*_{S_1/F_1} \simeq \sO_{C \times \tilde l}(-2,1)^{\oplus 2} 
$$
and
\begin{equation} \label{equation-conormal-1}
N_{S_1/Y_1}^* \simeq p_C^* W_C \otimes \sO_{C \times \tilde l}(0,1) 
\end{equation}
where $W_C$ is a rank three vector bundle on $C$ given by an extension
$$
0 \rightarrow \sO_C(-5) \rightarrow W_C \rightarrow \sO_C(-2)^{\oplus 2} \rightarrow 0.
$$
\end{lemma}

\begin{proof}
The expression for the conormal bundle $N^*_{S_1/F_1} = N^*_{\PP(\sO_{C \times \tilde l}(-5,1))/ \PP(N^*_{S/Y})}$ is immediate from Proposition \ref{prop:N2}.
For the computation of $N_{S_1/Y_1}^*$ consider the exact sequence
$$
0 \rightarrow N_{F_1/Y_1}^* \otimes \sO_{S_1} \simeq \sO_{C \times \tilde l}(-5,1) \rightarrow  N_{S_1/Y_1}^* \rightarrow N_{S_1/F_1}^* \simeq \sO_{C \times \tilde l}(-2,1)^{\oplus 2} \rightarrow 0
$$
and observe that $N_{S_1/F_1}^* \otimes \sO_{\tilde l} \simeq \sO_{\tilde l}(1)^{\oplus 3}$.
Note that $S_1 \simeq S = C \times \tilde l$, and let 
$p_C: S_1 \to C$ denote the projection. 
Twisting the exact sequence with $\sO_{C \times \tilde l}(0,-1)$ and pushing forward via $p_C$ we obtain an exact sequence
$$
0 \rightarrow \sO_C(-5) \rightarrow W_C := (p_C)_* (N_{S_1/F_1}^* \otimes \sO_{C \times \tilde l}(0,-1)) \rightarrow \sO_C(-2)^{\oplus 2} \rightarrow 0.
$$
Since the relative evaluation map
$$
p_C^* \left( (p_C)_* (N_{S_1/F_1}^* \otimes \sO_{C \times \tilde l}(0,-1)) \right)
\rightarrow N_{S_1/F_1}^* \otimes \sO_{C \times \tilde l}(0,-1)
$$
is surjective, this shows the statement.
\end{proof}

Using the notation of Proposition \ref{prop:N3} we set
$$
T_1 := \mathbb P(T) \subset \PP(N_{S/X}^*) = E_1,
$$
so that $T_1 \cap Y_1 = S_1$. 
In order to simplify the notation, we denote by $\holom{\mu_1}{T_1}{S_1}$
the restriction of $\mu_1$ to $T_1$.
Since $T \otimes \sO_{\tilde l} = \sO_{\tilde l}(1) \oplus \sO_{\tilde l}(1)$, we may write 
\begin{equation} \label{definitionUC}
T \simeq  p_C^*(U_C) \otimes \sO_{C \times \tilde l}(0,1) 
\end{equation}
with a rank 2-vector bundle $U_C \rightarrow C$ so that 
$ T_1 =  \PP(T) \simeq \PP(p_C^*(U_C))$ and we have a 
commutative diagram
\begin{equation} \label{bigdiagram1}
\xymatrix{
T_1  \simeq \PP(T) \simeq S_1 \times_C \PP(U_C) \ar[d]^{\mu_1} \ar[rd]^{q_F} & 
\\
S_1 \simeq C \times \tilde l \ar[d]^{p_C} & \PP(U_C) \ar[ld]^{\pi_C} & 
\\
C & & 
}
\end{equation}
Let 
$$
\holom{\mu_2}{X_2}{X_1}
$$ 
be the blow-up of $X_1$ along $T_1$, and denote by
$E_2 \subset X_2$ the exceptional divisor. For 
simplicity's sake denote by
$\holom{\mu_2}{E_2}{T_1}$
the restriction of $\mu_2$ to $E_2 \simeq \PP(N^*_{T_1/X_1})$. 

\begin{lemma} \label{lemma-E2}
The divisor $E_2$ is a $\PP^1$-bundle $\holom{\mu_3}{E_2}{W}$
over a fourfold $W$ that is a $\PP^2$-bundle $\holom{p_F}{W \simeq \PP(U_F)}{\PP(U_C)}$.

Moreover we have a rank two quotient bundle
$$
U_F \rightarrow \left(\pi_C^* \sO_C(-7)^{\oplus 2} \right) \otimes \zeta^*_{\PP(U_C)}
$$
such that
$$
\PP(N_{T_1/E_1}^*) \simeq T_1 \times_{\PP(U_C)} \PP(\left( \pi_C^* \sO_C(-7)^{\oplus 2} \right) \otimes \zeta^*_{\PP(U_C)})
$$
and the restriction of $\mu_3$ to $\PP(N_{T_1/E_1}^*) \subset E_2$ corresponds to the projection on the second factor.
\end{lemma} 

\begin{proof}
Since $T_1 = \PP(T) \subset \PP(N_{S/X}^*)=E_1$ we have
$N_{T_1/E_1}^* \simeq \left(\mu_1^* \sO_{C \times \tilde l}(-7,2)^{\oplus 2}\right) \otimes \zeta^*_{\PP(T)}$
by Proposition \ref{prop:N3}. Consider now
the exact sequence
$$
0 \rightarrow N_{E_1/X_1}^* \otimes \sO_{\PP(T)} \simeq \zeta_{\PP(T)} \rightarrow
N^*_{T_1/X_1} \rightarrow \left(\mu_1^* \sO_{C \times \tilde l}(-7,2)^{\oplus 2}\right) \otimes \zeta^*_{\PP(T)} \rightarrow 0.
$$
Since   $T \simeq p_C^*(U_C) \otimes \sO_{C \times \tilde l}(0,1)$ we have
$$
 \zeta_{\PP(T)} \simeq q_F^* \zeta_{\PP(U_C)} \otimes \mu_1^* \sO_{C \times \tilde l}(0,1),
$$
where $\zeta_{\PP(U_C)}$ is the tautological bundle on $\PP(U_C)$. Thus the exact sequence above simplifies to
$$
0 \rightarrow q_F^* \zeta_{\PP(U_C)} \otimes \mu_1^* \sO_{C \times \tilde l}(0,1) \rightarrow
N^*_{T_1/X_1} \rightarrow 
\left( \mu_1^* \sO_{C \times \tilde l}(-7,1)^{\oplus 2} \right) \otimes q_F^* \zeta^*_{\PP(U_C)} 
\rightarrow 0
$$
Twisting by $\mu_1^* \sO_{C \times \tilde l}(0,-1)$ and pushing forward via $q_F$ we argue as in the proof of Lemma \ref{lemma-normal-S1} to obtain
\begin{equation} \label{normal-T1}
N^*_{T_1/X_1} \simeq q_F^* U_F \otimes \mu_1^* \sO_{C \times \tilde l}(0,1)
\end{equation}
where $U_F \rightarrow \PP(U_C)$ is a rank three vector bundle given by an extension
$$
0 \rightarrow  \zeta_{\PP(U_C)} \rightarrow
U_F \rightarrow 
\left( \pi_C^* \sO_C(-7)^{\oplus 2} \right) \otimes \zeta^*_{\PP(U_C)} 
\rightarrow 0.
$$
Now the conclusion for $E_2$ immediate, since
$$
E_2 \simeq \PP(N^*_{\PP(T)/X_1}) \simeq \PP(q_F^* U_F) \simeq T_1 \times_{\PP(U_C)} \PP(U_F).
$$ 
For the statement about $\PP(N_{T_1/E_1}^*)$ just observe that
the isomorphism $N^*_{T_1/X_1} \simeq q_F^* U_F \otimes \mu_1^* \sO_{C \times \tilde l}(0,1)$
maps $N_{T_1/E_1}^*$ onto $q_F^* \left( (\pi_C^* \sO_C(-7)^{\oplus 2}) \otimes \zeta^*_{\PP(U_C)} \right)  \otimes \mu_1^* \sO_{C \times \tilde l}(0,1)$.
\end{proof}

We summarise the situation in a commutative diagram
\begin{equation} \label{bigdiagram2}
\xymatrix{
&  E_2 \simeq \PP(N^*_{T_1/X_1}) \simeq T_1 \times_{\PP(U_C)} \PP(U_F) \ar[ld]^{\mu_2} \ar[rd]^{\mu_3} &
\\
T_1  \simeq \PP(T) \ar[d]^{\mu_1} \ar[rd]^{q_F} & & W \simeq \PP(U_F) \ar[ld]^{p_F}
\\
S_1 \simeq C \times \tilde l \ar[d] & \PP(U_C) \ar[ld]^{\pi_C} & 
\\
C & & 
}
\end{equation}
Let $\bar l \simeq \PP^1$ be a fibre of the fibration $q_F$. Then it is not difficult to see that
$$
N_{\bar l/X_1}^* \simeq \sO_{\bar l}^{\oplus 2} \oplus \sO_{\bar l}(1)^{\oplus 3},
$$ 
the trivial part corresponding to the pull-back of $T_{\PP(U_C)}$. It is well-known
that a rational curve $\PP^1 \subset Q$ in a smooth fourfold $Q$
with $N^*_{\PP^1/Q} \simeq \sO_{\PP^1}(1)^{\oplus 3}$ can be flipped, i.e. there exists
a smooth fourfold $Q^-$ containing a $\PP^2$ with $N_{\PP^2/Q^-}^* \simeq \sO_{\PP^2}(1)^{\oplus 2}$
and $Q \setminus \PP^1 \simeq Q^- \setminus \PP^2$. Indeed this is the inverse
of the smooth fourfold flip described by Kawamata \cite[Thm.1.1]{Kaw89}.
The next statement shows that in our situation we can construct a relative version of Kawamata's antiflip:

\begin{proposition}
The fibration $\holom{\mu_3}{E_2}{W}$ extends to a bimeromorphic map
$$
\holom{\mu_3}{X_2}{X_3}
$$
of compact complex manifolds such that $X_2$ is the blow-up of $X_3$ along $W$.
\end{proposition}

\begin{remark*} At this point we do not claim that $X_3$ is a projective manifold, although this will follow from Section 5. 
\end{remark*}

\begin{proof}
Fujiki's criterion \cite[Thm.1]{Fuj75} guarantees the existence of $\mu_3$ if we check that the conormal bundle $N_{E_2/X_2}^*$
is ample on the fibres of the $\PP^1$-bundle $\mu_3$. We will prove that $N_{E_2/X_2}^*$ has degree one on the fibres, so the bimeromorphic map is a blow-up of the manifold $X_3$ along $W$.

Recall first that by \eqref{normal-T1} we have
$$
N^*_{T_1/X_1} \simeq q_F^* U_F \otimes \mu_1^* \sO_{C \times \tilde l}(0,1)
$$
Since $N_{E_2/X_2}^* \simeq \zeta_{\PP(N^*_{T_1/X_1})}$ we obtain
$$
N_{E_2/X_2}^* \simeq \mu_3^* \zeta_{U_F} \otimes \mu_2^* \mu_1^* \sO_{C \times \tilde l}(0,1).
$$
Since $E_2 \simeq T_1 \times_{\PP(U_C)} W$, the $\mu_3$-fibres map isomorphically onto the $q_F$-fibres. 
Since $T_1 \simeq S_1 \times_C \PP(U_C)$ the $q_F$-fibres map isomorphically onto the $p_C$-fibres.  
Yet $\mu_1^* \sO_{C \times \tilde l}(0,1)$ has degree one on the $p_C$-fibres, so the statement follows.
\end{proof}

\begin{proposition} \label{proposition-strict-transform}
Denote by $Y_3 \subset X_3$ the strict transform of $Y_1 \subset X_1$ under the ``Kawamata anti-flip'' $\mu_3 \circ \mu_2^{-1}$. Then $Y_3$ does not contain $W$.
\end{proposition}

\begin{proof}
Let $Y_2 \subset X_2$ be the strict transform of $Y_1$. The statement is equivalent to showing that $Y_2$ is disjoint from a general exceptional curve of $\mu_3$. Since $T_1$ is not contained in $Y_1$ we have $Y_2 \simeq \mu_2^* Y_1$. Since the $\mu_3$-fibres map onto $q_F$-fibres it is sufficient to show that $Y_1 \cdot \bar l=0$ where $\bar l \subset T_1$ is a $q_F$-fibre. By definition, cf. \eqref{surf}, the surface $S_1$ is given by the quotient 
$N^*_{S/X} \rightarrow N^*_{S/Y} \rightarrow \sO_{C \times \tilde l}(-5,1)$.
By Proposition \ref{prop:N3} the quotient map $N^*_{S/X} \rightarrow \sO_{C \times \tilde l}(-5,1)$
factors into 
$$
N^*_{S/X} \rightarrow T \rightarrow \sO_{C \times \tilde l}(-5,1),
$$
so we have $Y_1 \cap T_1 = S_1$ and we are left to show that $S_1$ is disjoint from the general $q_F$-fibre.
By \eqref{definitionUC} we have 
$$
T \simeq p_C^* U_C \otimes \sO_{C \times \tilde l}(0,1)
$$
and $U_C \rightarrow C$ is given by an extension
$$
0 \rightarrow \sO_C(-7) \rightarrow U_C \rightarrow \sO_C(-5) \rightarrow 0.
$$
Hence we obtain that the isomorphism $T_1=\PP(T) \rightarrow \PP(p_C^* U_C)$
maps the surface $S_1=\PP(\sO_{C \times \tilde l}(-5,1))$ onto  
$$
\PP(p_C^* \sO_C(-5)) \simeq q_F^* \PP(\sO_C(-5))
$$
where in the last step we consider $\PP(\sO_C(-5))$ as a divisor in $\PP(U_C)$.
In particular its intersection number with a general $q_F$-fibre is zero.
\end{proof}

\begin{proposition} \label{proposition-p3-bundle}
The strict transform $\tilde E_1 \subset X_3$ of $E_1$ has a $\PP^3$-bundle structure  $$
\holom{\mu_4}{\tilde E_1}{C \times D},
$$
where $D \simeq \PP^1$.
Moreover the fourfold $W$ is not contained $\tilde E_1$.
\end{proposition}

\begin{proof}
We start considering the strict transform $E_1' \subset Y_2$, i.e., the blowup of $E_1$ along $S_1$.
Since $E_1 \simeq \PP(N_{S/X}^*)$ and since $T_1 \subset E_1$ corresponds
to the quotient
$$
0 \rightarrow K := \sO_{C \times \tilde l}(-7,2)^{\oplus 2} \rightarrow N_{S/X}^* \rightarrow T \rightarrow 0,
$$
the blowup $E_1'$ has a $\PP^2$-bundle structure $\holom{\psi}{E_1'}{\PP(K)}$. This $\PP^2$-bundle structure corresponds to a rank three vector bundle $Q$ given by an extension
\begin{equation} \label{define-Q}
0 \rightarrow \sO \rightarrow Q \rightarrow  p^* T \otimes \zeta_{\PP(K)}^* \rightarrow 0,
\end{equation}
where $\holom{p}{\PP(K)}{S}$ is the natural map and $\zeta_{\PP(K)}$ the tautological divisor.

Denote by $N \simeq \PP(N^*_{T_1/E_1})$  the exceptional divisor 
of $E_1' \rightarrow E_1$.
With the notation of \eqref{define-Q} the exceptional divisor corresponds to the quotient
$Q \rightarrow p^* T \otimes \zeta_{\PP(K)}^*$. 

Note also that since $K \simeq \sO_{C \times \tilde l}(-7,2)^{\oplus 2}$, we have
$$
\PP(K) \simeq C \times \tilde l \times D,
$$
with $D \simeq \PP^1$.
Hence we have fibration
$$
\tau : E_1' \rightarrow \PP(K) \rightarrow C \times D
$$
such that the fibres are $\PP^2$-bundles over $\tilde l$.
Restricting \eqref{define-Q} to $\tilde l$ we see that
$$
Q \otimes \sO_{\tilde l} \simeq \sO_{\tilde l} \oplus \sO_{\tilde l}(-1)^{\oplus 2},
$$ 
so the fibres of $\tau$ are isomorphic to the blow-up of $\PP^3$ along a line,
and the exceptional divisor of $\PP(Q \otimes \sO_{\tilde l}) \rightarrow \PP^3$
is given by the second projection of 
$$
N \cap \fibre{\psi}{\tilde l} \simeq \PP (p^* T \otimes \zeta_{\PP(K)}^* \otimes \sO_{\tilde l}) \simeq \PP(\sO_{\tilde l}(-1)^{\oplus 2}) \simeq \tilde l \times \PP^1.
$$
Yet these are exactly the fibres of the fibration
$$
N \rightarrow \PP(\left( \pi_C^* \sO_C(-7)^{\oplus 2} \right) \otimes \zeta^*_{\PP(U_C)})
$$
given by Lemma \ref{lemma-E2}, i.e. the restriction of $\mu_3$ to $N$. Since the divisor $N \subset E_1'$
has degree minus one on these curves, we can again apply Fujiki's criterion
\cite[Thm.1]{Fuj75} to obtain a birational morphism
$$
E_1' \rightarrow Z \rightarrow C \times d
$$
such that $\mu_4: Z \rightarrow C \times d$ is a $\PP^3$-bundle. Since
the fibres of  $E_1' \rightarrow Z$ coincide with the fibres of $E_1' \rightarrow \tilde E_1$, we have $Z \simeq \tilde E_1$. 

Finally let us show that $W$ is not contained in $\tilde E_1$. Since
$\mu_3$ is an isomorphism in the complement of $E_2$, 
this comes down to show that 
$$
\mu_3(E_1' \cap E_2 ) = \mu_3(N) = \mu_3(\mathbb P(N^*_{T_1/E_1})) 
$$
is a proper subset of $W$. Yet by Lemma \ref{lemma-E2} the birational
map $\mu_3$ maps the fourfold $N$ onto the threefold
$\PP(\left( \pi_C^* \sO_C(-7)^{\oplus 2} \right) \otimes \zeta^*_{\PP(U_C)})$.
Thus it is a proper subset of the fourfold $W$.
\end{proof}

Finally we construct $\mu_4: X_3 \to X_4$.

\begin{proposition} \label{proposition-mu4}
The restriction of the normal bundle $N_{\tilde E_1/X_3} $ to the fibres of
the $\PP^3$-bundle $\mu_4: \tilde E_1 \to C \times D$ is isomorphic
to $\sO_{\PP^3}(-2)$.

In particular, by Fujiki's criterion \cite[Thm.1]{Fuj75},
there exists a bimeromorphic morphism $\mu_4: X_3 \to X_4$ to a normal compact complex space $X_4$ contracting $\tilde E_1$ onto $C \times D$. 
\end{proposition} 

\begin{proof} 
We use the notation from the proof of Proposition \ref{proposition-p3-bundle}.
Since $W$, the center of the blow-up $\mu_3$, is not contained in $\tilde E_1$, we have
$$ 
E'_1 = \mu_3^* \tilde E_1.
$$ 
Let $C \subset E_1'$ be a line in a fibre of the $\PP^2$-bundle $E'_1 \to \PP(K)$. Then $\mu_3$ maps $C$ isomorphically to a line in a fibre of $\mu_4$.  Hence it suffices to show that $E'_1 \cdot C = -2$. 
Now 
$$ E_1' = \mu_2^*E_1  - E_2,$$
and $E_1' \cap E_2 = N $ is a $\PP^1$-subbundle of the $\PP^2$-bundle
$E_1' \to \PP(K)$. Hence $N \cdot C = 1$, and we are left to show that 
$$ 
\mu_2^*E_1 \cdot C = E_1 \cdot \mu_2(C) = -1.
$$ 
The curve $\mu_2(C)$ is a line in a fibre of the $\PP^2$-bundle $E_1 \to S_1$, and $E_1$ is the exceptional locus of the 
blow-up of $X$ along $S_1$. Hence $E_1 \cdot \mu_2(C) = -1$. 
\end{proof}

\begin{lemma} \label{lemma-construct-Mori-contraction}
Denote by $X'$ the Stein factorisation of the image of the morphism
$\holom{\varphi_V}{X=\PP(V)}{\PP_M}$. 
There exists a bimeromorphic
morphism
$$
\holom{\varphi^-}{X_4}{X'}
$$
such that the exceptional locus coincides with $\mu_4(W) \cup (C \times D)$.
More precisely the restriction of $\varphi^-$ to $C \times D$ corresponds
to the projection onto $C$, and the restriction to $\mu_4(W)$ corresponds to
the fibration $\pi_C \circ p_F$ (cf. Diagram \eqref{bigdiagram2}).
\end{lemma}

\begin{proof}
Consider the birational morphism
$$
\tau: X_2 \stackrel{\mu_1 \circ \mu_2}{\longrightarrow} X \stackrel{\varphi_V}{\longrightarrow} X',
$$
and recall that $\varphi_V$ simply contracts $S \simeq C \times \tilde l$ onto $C$.
By the rigidity lemma this morphism factors through $\holom{\mu_4 \circ \mu_3}{X_2}{X_4}$ if we show that the $(\mu_4 \circ \mu_3)$-fibres are contained in $\tau$-fibres.
The exceptional locus of $\mu_4 \circ \mu_3$ is equal to $E_1' \cup E_2$ and the restriction of $\mu_4 \circ \mu_3$ to the exceptional divisors is given by
$$
E_1' \stackrel{\mu_3}{\longrightarrow} \tilde E_1 \stackrel{\mu_4}{\longrightarrow} C \times D
$$
and 
$$
\mu_3: E_2 \rightarrow W.
$$
By \eqref{bigdiagram2} and \eqref{bigdiagram1} the $\mu_3$-fibres map onto $p_C$-fibres, so for $E_2$ the property is clear. For $E_1'$ recall from the proof of Proposition \ref{proposition-p3-bundle} that we have a $\PP^2$-bundle structure
$E_1' \rightarrow C \times \tilde l \times D$ so that the fibres of $E_1' \rightarrow C \times D$ are $\PP^2$-bundles over $\tilde l$. The restriction of $\tau$ to $E_1'$
coincides with 
$$
E_1' \rightarrow C \times \tilde l \times D \rightarrow C,
$$
so we obtain the desired property for $E_1'$.
\end{proof}

\section{Proofs of the main results}

\begin{proof}[Proof of Theorem \ref{theorem-V5}]
We use the notation introduced in Subsection \ref{subsectionPV}. 
Although the birational map 
$$
\mu_3 \circ (\mu_1 \circ \mu_2)^{-1} : X \dashrightarrow X_3
$$
is not an isomorphism in codimension one, we can show that $Z_M$ is not Stein. In fact,
we have
$$ 
Z_M \simeq X \setminus Y \simeq X_2 \setminus (Y_2 \cup E_1' \cup E_2) \simeq X_3 \setminus (Y_3 \cup \tilde E_1 \cup W).
$$
By Proposition \ref{proposition-strict-transform} and \ref{proposition-p3-bundle}
the codimension two subset $B$ is not contained in the divisors $Y_3 \cup \tilde E_1$. 
Thus by \cite[V, \S 1, Thm.4]{GR04} the complex space $Z_M$ is not Stein.
\end{proof}

\begin{proof}[Proof of Theorem \ref{theorem-V5-details}]
We use the notation introduced in Subsection \ref{subsectionPV}. 
We have already seen in the proof of Theorem \ref{theorem-V5} that $Z_M$ is not Stein.
Set now
$$
X^- := \PP(V)^- := X_4,
\qquad
\psi := (\mu_4 \circ \mu_3) \circ (\mu_1 \circ \mu_2)^{-1} : X \dashrightarrow X^-
$$
and denote by $Y^- := \PP(T_M)^-:=Y_4$ the strict transform of $Y=\PP(T_M)$ in $X_4$.

By construction of $\mu_1$ and $\mu_2$, the map
$$
X_2 \setminus (E_1' \cup E_2) \rightarrow X \setminus S
$$ 
is an isomorphism. 
By construction of $\mu_3$ and $\mu_4$, the map
$$
X_2 \setminus (E_1' \cup E_2) \rightarrow X^- \setminus (\mu_4(W) \cup (C \times D))
$$ 
is an isomorphism. Thus it is clear that $\psi$ is an isomorphism in codimension one.

Set $A=\mu_4(W) \cup (C \times D)$. Since $W$ is not contained in $Y_3$ by
Proposition \ref{proposition-strict-transform} and since $\mu_4$ is an isomorphism
at the generic point of $W$ by Proposition \ref{proposition-p3-bundle},
the fourfold $\mu_4(W)$ is not contained in $Y^-$. This shows statement $a)$.

By what precedes the bimeromorphic map $\psi$ induces an isomorphism
$$
Y \setminus S \rightarrow Y^- \setminus (A \cap Y^-).
$$
Since $\mu_4(B)$ is not contained in $Y^-$, the intersection $A \cap Y^-$ has codimension at least two in $Y^-$. Thus $Y \dashrightarrow Y^-$
is an isomorphism in codimension one. 
To show that $X \dasharrow X^-$ is antiflip as well as $Y \dasharrow Y^-$, it suffices that $-K_{Y^-}$ is $\varphi^-|_{Y^-}$-ample. This is shown in Proposition 
\ref{proposition-final} below. Thus 
 statement (b) is established.

For the proof of $c)$, note that by construction it is clear that both spaces $X^-$ and $Y^- $ are $\mathbb Q$-factorial with terminal singularities.
Since $-K_{X^-}$ is $\varphi^-$-ample by b), the bimeromorphic morphism $\varphi^- : X^- \rightarrow X'$
is projective. Since $X'$ is a projective variety, the complex space $X^-$ is thus projective.
Thus it suffices to show that the anticanonical divisors are big and nef. 
Since $\psi$ and $\psi|_Y$ are isomorphisms in codimension one
we have
$$
-K_{X^-} = \psi_* (-K_{X}) = \psi_* (4 c_1(\zeta_V)) 
$$
and
$$
-K_{Y^-} = (\psi|_Y)_* (-K_{Y}) = (\psi|_Y)_* (3 c_1(\zeta_M)). 
$$
Since $T_M$ is big, the bigness of the anticanonical divisor is thus clear and we are left to show nefness. Since 
$$ 
Y^- \in \vert \frac{1}{4} (-K_{X^-}) \vert, 
$$ 
it suffices to show that $-K_{Y^-}$ is nef. 
This is settled in Proposition \ref{proposition-final}. 

Finally, statement (d) follows from \cite[Thm.1.2]{HP21}. 
\end{proof}

\begin{proposition} \label{proposition-final}
The anticanonical divisor of $Y^-$ is nef.
Furthermore, the anticanonical divisor of $Y^-$ is relatively ample with respect
to the morphism $\varphi^-$.
\end{proposition} 

\begin{proof} 
We set $\nu_j := {\mu_j}_{\vert Y_j}$,
and
$$
F'_1 := E'_1 \cap Y_2, \qquad  \tilde F_1 := \tilde E_1\cap Y_3.
$$ 
It is not hard to check that all the intersections are transversal,
in particular we have $\nu_2^* F_1 = F_1'+F_2$.

Set $g = \nu_1 \circ \nu_2: Y_2 \to Y$ and $h = \nu_4 \circ \nu_3: Y_2 \to Y^-=Y_4$. 
Since $S_1$  is  contained in $F_1$, we have
$$ 
-K_{Y_2} = g^*(-K_Y)  - 2 F'_1 - 4F_2.
$$
Since $\nu_3$ is a smooth blowup and along a center that is not contained in $\tilde F_1$ (Proposition \ref{proposition-p3-bundle}) and $\nu_3$ contracts the $\PP^2$-bundle 
$\tilde F_1 \rightarrow C \times D$ onto $C \times D$ with $N_{\tilde F_1/Y_3} \otimes \sO_{\PP^2} \simeq \sO_{\PP^2}(-2)$,
we have
$$ 
-K_{Y_2} = h^*(-K_{Y_4} ) - \frac{3}{2} F'_1 - F_2.
$$
Thus we obtain
\begin{equation} \label{eq:base} 
h^*(-K_{Y_4}) = g^*(-K_Y) - \frac{3}{2}F'_1 - 3F_2 = g^*(-K_Y) - 3 \nu_2^* F_1 + \frac{3}{2} F_1'.
\end{equation} 
Recall first that $-K_Y = 3 \zeta_M$ and $\Bs{(|\zeta_M|)} = S \cup Z$ where 
$Z := \PP(N_{B/M})$. Thus we have
$$
\Bs{((\psi_Y)_*|\zeta_M|)} \subset h(F_1' \cup F_2 \cup Z_2) 
$$
where $Z_2 \subset Y_2$ is the strict transform of $Z$. 

{\em Step 1. $h^*(-K_{Y_4})$ is nef on $Z_2$.}
Since $N_{B/M}  = \sO_B(5) \oplus \sO_B(5)$ by \eqref{normal-B}, we have $Z \simeq B \times \PP^1$. Moreover since $N^*_{B/M} \simeq N_{B/M} \otimes \sO_B(-10)$,
the blowup of $M$ along $B$ embeds into $Y=\PP(T_M)$. 
Thus we can apply \cite[Lemma 7.2.3(iii)]{CS16} to see that
the intersection $Z \cap S$ is not transversal and $[Z \cap S]=2 \Delta$
where $\Delta \in |\sO_{B \times \PP^1}(1,1)|$.

Since $\nu_1$ is the blowup of $Y$ along $S$, the intersection of the strict transform
$Z_1 \subset Y_1$ with the exceptional divisor $F_1$ has still support along $\Delta$,
yet now the intersection is transversal. 
Thus we obtain that 
\begin{equation} \label{delta}
\sO_{Z_2}(F_1) \simeq \sO_{B \times \PP^1}(1,1).
\end{equation}
We claim that 
$$
F_1 \cap Z_1 = S_1 \cap Z_1.
$$
Indeed by \eqref{surf} the surface $S_1 \subset F_1$ corresponds to the quotient $N_{S/Y}^* \rightarrow \sO_{C \times \tilde l}(-5,1)$,
 which is the quotient determined by the direct factor
$\sO_{C \times \tilde l}(2,0)$, i.e. the relative tangent bundle of the projection $S \simeq C \times \tilde l \rightarrow \tilde l$. Yet these curves are tangent to $Z$ along $\Delta$, so
$\PP(\sO_{C \times \tilde l}(-5,1) \otimes \sO_\Delta)$ is contained in $Z_1$. 
This shows the claim.

Since $\nu_2$ is the blow-up of $Y_1$ along $S_1$ and $S_1 \cap Z_1 = F_1 \cap Z_1 \subset Z_1$ is a Cartier divisor, the
strict transform $Z_2 \subset Y_2$ is isomorphic to  $Z_1 \simeq Z \simeq B \times \PP^1$.
Now we can finally compute the restriction:
since $N_{B/M}  \simeq \sO_B(5) \oplus \sO_B(5)$, we have
$$
\sO_Z(-K_Y)  \simeq \zeta_M^{\otimes 3} \otimes \sO_Z \simeq \sO_{\PP(N_{B/M})}(3) \otimes \sO_Z \simeq \sO_{B \times \PP^1}(15,3).
$$
Hence $\sO_{Z_2}(g^*(-K_Y)) \simeq \sO_{B \times \PP^1}(15,3)$. 
By \eqref{delta} we have $\sO_{Z_2}(\nu_2^* F_1) \simeq \sO_{B \times \PP^1}(1,1)$, hence
 $$ 
\sO_{Z_2}(g^*(-K_Y) - 3 \nu_2^* F_1 + \frac{3}{2} F_1')
\sim_\Q \sO_{B \times \PP^1}(12,0) \otimes \sO_{Z_2}(\frac{3}{2} F'_1),
$$
is nef, since any effective divisor on $Z_2 \simeq B \times \PP^1$ is nef. 
By \eqref{eq:base} this finishes the proof of Step 1.

{\em Step 2. $h^*(-K_{Y_4})$ is nef on $F_2$.}
Since $\nu_2$ is the blowup of $S_1 \subset Y_1$ and $S_1 \simeq S$ we obtain from
Lemma \ref{lemma-zetaM} that 
$$ 
\sO_{F_2}(g^*(-K_Y)) \simeq g^* \sO_{C \times \tilde l}(21,-3).
$$
Furthermore $F_2 \simeq \PP(N_{S_1/Y_1}^*)$ and $F_1' \cap F_2 \simeq \PP(N_{S_1/F_1}^*)$ , so by \eqref{surf}
$$ 
\sO_{F_2}(F'_1) \simeq 
\sO_{\PP(N_{S_1/Y_1}^*)}(1) \otimes \nu_2^* \sO_{C \times \tilde l}(5,-1)
$$
Since $\sO_{F_2}(F_2) \simeq \sO_{\PP(N_{S_1/Y_1}^*)}(-1) $ we obtain from 
\eqref{eq:base} that
$$
\sO_{F_2}(2 h^*(-K_{Y_4})) = (\sO_{\PP(N_{S_1/Y_1}^*)}(1)  \otimes \nu_2^* \sO_{C \times \tilde l}(9,-1))^{\otimes 3}.
$$
Yet by \eqref{equation-conormal-1} we have $N_{S_1/Y_1}^* \simeq p_C^* W_C \otimes \sO_{C \times \tilde l}(0,1)$, so
$$
N_{S_1/Y_1}^* \otimes \sO_{C \times \tilde l}(9,-1) \simeq p_C^* W_C \otimes \sO_{C \times \tilde l}(9,0)
$$
is nef by the extension defining $W_C$ (cf. Lemma \ref{lemma-normal-S1}).

Note that this also shows that $\mu_3$ maps $F_2$ onto $\PP(W_C)$ and $\PP(W_C) \subset Y_4$
is one of the two irreducible components of the exceptional locus of $\varphi^-|_{Y_4}$.
Since $-K_{Y_4}$ is a positive multiple of the tautological class along the fibres of
$\PP(W_C) \rightarrow C$, we obtain the first half of the second statement.

{\em Step 3. $h^*(-K_{Y_4})$ is nef on $F_1'$.}
We proceed analogously to the proof of Proposition \ref{proposition-p3-bundle}:
since $F_1 \simeq \PP(N_{S/Y}^*)$ and since $S_1 \subset F_1$ corresponds
to the quotient
$$
0 \rightarrow K := \sO(-7,2)^{\oplus 2} \rightarrow N_{S/Y}^* \rightarrow \sO(-5,1) \rightarrow 0,
$$
the blowup $F_1'$ has a $\PP^1$-bundle structure $\holom{\psi}{E_1'}{\PP(K)}$. This $\PP^2$-bundle structure corresponds to a rank two vector bundle $Q_Y$ given by an extension
\begin{equation} \label{define-QY}
0 \rightarrow \sO \rightarrow Q_Y \rightarrow  p^* \sO(-5,1) \otimes \zeta_{\PP(K)}^* \rightarrow 0,
\end{equation}
where $\holom{p}{\PP(K)}{S}$ is the natural map and $\zeta_{\PP(K)}$ the tautological divisor.
The exceptional divisor $N$ of the blow-up $F_1' \rightarrow F_1$, i.e. the intersection
$F_2 \cap F_1$, is given by the quotient $Q \rightarrow  p^* \sO(-5,1) \otimes \zeta_{\PP(K)}^*$.
Thus we have 
$$
F_2 \cap F_1 = N.
$$
Note also that $\nu_2^* \zeta_{\PP(N_{S/Y}^*)}(1) \simeq p^* \zeta_{\PP(K)} \otimes \sO_{F_1'}(N)$,
so 
$$
\sO_{F_1'}(\nu_2^* F_1) \simeq \nu_2^* \zeta_{\PP(N_{S/Y}^*)}(-1) \simeq p^* \zeta^*_{\PP(K)} 
\otimes \sO_{F_1'}(-N).
$$
Thus by \eqref{eq:base} we have
$$
\sO_{F_1'}(2 h^*(-K_{Y_4})) \simeq \sO_{F_1'}(2 g^*(-K_Y) - 3 \nu_2^* F_1 - 3 F_2)
\simeq \sO_{F_1'}(2 g^*(-K_Y)) \otimes p^* \zeta^{\otimes 3}_{\PP(K)}.
$$
Since $K \simeq \sO(-7,2)^{\oplus 2}$ and $\sO_S(-K_Y) = \sO_{C \times \tilde l}(21,-3)$, we see that
$$
\sO_{F_1'}(2 g^*(-K_Y)) \otimes p^* \zeta^{\otimes 3}_{\PP(K)}
\simeq \sO_{\PP(\sO_{C \times \tilde l}(7,0)^{\oplus 2})}(3)
$$
is nef. This equation also shows that the restriction of $-K_{Y_4}$ to $C \times D$
is isomorphic to $\sO_{C \times D}(21,3)$. Since $C \times D$
is one of the two irreducible components of the exceptional locus of $\varphi^-|_{Y_4}$,
we obtain the second half of the second statement.

{\em Step 4. Conclusion.}
By Step 2 and 3, the morphism $\varphi^-$ is projective, in
fact $-K_{Y^-}$ is relatively ample. Since the image of $\varphi^-$ is a projective variety, the complex space $Y_4=Y^-$
is thus projective. In particular we can verify the nefness
of the anticanonical divisor by computing the intersection with
curves in $Y_4$.

Since
$-K_{Y_4} \simeq 3 (\psi_Y)_* \zeta_M$, it is sufficient to show that $-K_{Y_4}$
is nef on $h(F_1' \cup F_2 \cup Z_2)$. 
Yet by the Steps 1-3 the divisor $h^*(-K_{Y_4})$ is nef on  $F_1' \cup F_2 \cup Z_2$, this finishes the proof.
\end{proof}

\begin{proof}[Proof of Theorem \ref{thm:Fano3}]

Suppose first that $\rho(M) \geq 2$. By \cite[Cor.1.8]{HP21}, any contraction $f: M \to N$  of an extremal ray is a fibre space. 
Further, by Corollaries \ref{cor1} and \ref{cor2}, the fibration $f$ is either
a $\mathbb P^1$-bundle over the surface $N$, 
a $\mathbb P^2$-bundle or a quadric bundle over $\mathbb P^1$. Going through the classification \cite{MM81}, \cite{MM82}, \cite{MM03}, only the following cases remain:
\begin{enumerate}
\item $M = \mathbb P(T_{\mathbb P^2})$;
\item $M = \mathbb P^2 \times \mathbb P^1$;
\item $M = \mathbb P^1 \times \mathbb P^1 \times \mathbb P^1$.
\end{enumerate}

We are therefore reduced to the case $\rho(M) = 1$. Since $Z_M$ is affine, the tangent bundle $T_M$ is big \cite[Prop.4.2]{GW20}. Hence 
we know by \cite[Cor.1.2]{HL21} that $M \simeq \PP^3, M \simeq Q^3$ or
$M = V_5$, the del Pezzo threefold of degree $5$. 
This last case  is excluded by Theorem \ref{theorem-V5-details}. 
\end{proof}


\providecommand{\bysame}{\leavevmode\hbox to3em{\hrulefill}\thinspace}
\providecommand{\MR}{\relax\ifhmode\unskip\space\fi MR }
\providecommand{\MRhref}[2]{%
  \href{http://www.ams.org/mathscinet-getitem?mr=#1}{#2}
}
\providecommand{\href}[2]{#2}

\end{document}